\documentclass[12pt,a4paper]{amsart}
\usepackage{amsfonts}
\usepackage{amssymb}
\usepackage{amsmath}
\usepackage{amsthm}
\usepackage{cite}
\usepackage{enumitem}
\usepackage{tikz}
\usepackage{subfigure}

\theoremstyle{plain}
\newtheorem{thm}{Theorem}

\newtheorem{prop}{Proposition}
\newtheorem{cor}{Corollary}
\newtheorem{bem}{Remark}

\providecommand{\sm}{\setminus}
\providecommand{\N}{\mathbb{N}}
\providecommand{\R}{\mathbb{R}}

\providecommand{\eps}{\varepsilon}

\providecommand{\dx}{\,dx}
\providecommand{\wto}{\rightharpoonup}
\providecommand{\skp}[2]{\langle#1,#2\rangle}
\DeclareMathOperator{\supp}{supp}

\renewcommand{\qed}{\hfill $\Box$}

\setitemize{itemsep=+2pt}
\setenumerate{itemsep=+2pt}

\begin{document}

\allowdisplaybreaks

\title{Minimal energy solutions for repulsive nonlinear Schr\"odinger~systems}

\author{Rainer Mandel}
\address{R. Mandel \hfill\break
Department of Mathematics, Karlsruhe Institute of Technology (KIT)\hfill\break
D-76128 Karlsruhe, Germany}
\email{Rainer.Mandel@kit.edu}
\date{15.03.2013}


\begin{abstract}
  In this paper we establish existence and nonexistence results concerning fully nontrivial minimal energy
  solutions of the nonlinear Schr\"odinger system
  \begin{align*}
     \begin{gathered}
       -\Delta u + ~~\, u = |u|^{2q-2}u + b|u|^{q-2}u|v|^q \quad\text{in }\R^n, \\
       -\Delta v + \omega^2 v = |v|^{2q-2}v + b|u|^q|v|^{q-2}v\quad\text{in }\R^n.
     \end{gathered}
  \end{align*}
  We consider the repulsive case $b<0$ and assume that the exponent $q$ satisfies $1<q<\frac{n}{n-2}$ in case
  $n\geq 3$ and $1<q<\infty$ in case $n=1$ or $n=2$. For space dimensions $n\geq 2$ and arbitrary $b<0$ we prove the
  existence of fully nontrivial nonnegative solutions which converge to a solution of some optimal partition
  problem as $b\to -\infty$. In case $n=1$ we prove that minimal energy solutions exist provided the coupling
  parameter $b$ has small absolute value whereas fully nontrivial solutions do not exist if $1<q\leq 2$ and
  $b$ has large absolute value.
\end{abstract}

\maketitle

 \section{Introduction}

 \parindent0mm

 In this paper we are interested in fully nontrivial minimal energy solutions of the system
  \begin{align}\label{Mbrad Gl DGL}
     \begin{gathered}
       -\Delta u + ~~\, u = |u|^{2q-2}u + b|u|^{q-2}u|v|^q \quad\text{in }\R^n, \\
       -\Delta v + \omega^2 v = |v|^{2q-2}v + b|u|^q|v|^{q-2}v\quad\text{in }\R^n
     \end{gathered}
  \end{align}
  for parameter values $\omega\geq 1$ and $b\leq 0$. We henceforth assume that the exponent $q$ satisfies
  $1<q<\frac{n}{n-2}$ when $n\geq 3$ and $1<q<\infty$ when $n=1$ or $n=2$.
  For applications in physics the special case $q=2$ and $n\in\{1,2,3\}$ is of particular
  importance. For example, in photonic crystals the system \eqref{Mbrad Gl DGL} is used to describe the
  approximate shape of so-called band gap solitons which are special nontrivial solitary wave solutions $E(x,t)=e^{-i\kappa
  t}\phi(x)$ of the time-dependent nonlinear Schr\"odinger equation (or Gross-Pitaevski
  equation)
  $$
    i\partial_t E = -\Delta E + V(x) E - |E|^2 E \qquad\text{in }[0,\infty)\times \R^n.
  $$
  For a detailed exposition on that matter we refer to \cite{DoUe_Coupled_mode_equations}.
  \medskip

  During the last ten years many authors contributed to a better unterstanding of such nonlinear
  Schr\"{o}dinger systems and various interesting results concerning the existence of nontrivial solutions
  have been proved using Ljusternik-Schnirelman theory~\cite{WeWe_Radial_solutions_and}, constrained minimization methods
  \cite{AmbCol_Bound_and_Ground},\cite{DeFLop_Solitary_waves_for},\cite{MaMoPe_Positive_solutions_for},\cite{Sir_Least_energy_solutions},\cite{WeWe_Nonradial_symmetric_bound}
  or bifurcation theory \cite{BaDaWa_A_Liouville_Theorem}. In the case of a positive coupling parameter $b$
  many existence results for positive solutions of \eqref{Mbrad Gl DGL} have been proved by investigations of
  appropriate constrained minimization problems. For instance Maia, Montefusco, Pellacci
  \cite{MaMoPe_Positive_solutions_for} proved the existence of nonnegative ground states of \eqref{Mbrad Gl DGL} which, by definition, are solutions of
  minimal energy among all nontrivial solutions. Here, the energy corresponds to the Euler functional~$I$
  associated to \eqref{Mbrad Gl DGL} which is given by
  $$
    I(u,v)
    = \frac{1}{2} \big( \|u\|^2 + \|v\|_\omega^2 \big)
      - \frac{1}{2q} \big(\|u\|_{2q}^{2q} + \|v\|_{2q}^{2q} +2b\|uv\|_q^q \big)
  $$
  where $\|\cdot\|_{2q},\|\cdot\|_q$ denote Lebesgue norms and $\|\cdot\|,\|\cdot\|_\omega$ denote
  Sobolev space norms that we will define in \eqref{Gl Defn norms}. Moreover the authors gave
  sufficient conditions and necessary conditions for ground states to be positive in both components which basically require the coupling parameter $b$ to be positive and
  sufficiently large. In the special case $q=2$ additional sufficient conditions for the existence of positive
  ground states have been proved in \cite{AmbCol_Bound_and_Ground},\cite{DeFLop_Solitary_waves_for}. Furthermore, for
  $q=2$ and small positive values of $b$ Lin, Wei
  \cite{LiWe_Ground_states_of},\cite{LiWe_Erratum_ground_states} and Sirakov~\cite{Sir_Least_energy_solutions}
  proved the existence of positive solutions which have minimal energy among all fully nontrivial solutions.
  From a technical point of view the approaches followed in \cite{MaMoPe_Positive_solutions_for} and
  \cite{Sir_Least_energy_solutions},\cite{LiWe_Ground_states_of},\cite{LiWe_Erratum_ground_states}
  differ in the following way. In \cite{MaMoPe_Positive_solutions_for} ground states are obtained by
  minimizing the Euler functional~$I$ over the entire Nehari manifold
  $$
    \mathcal{N}_b
    = \Big\{ (u,v):\; u,v\in H^1(\R^n),\;(u,v)\neq (0,0),\;\|u\|^2+\|v\|_\omega^2 =
    \|u\|_{2q}^{2q}+ \|v\|_{2q}^{2q}+2b\|uv\|_q^q\Big\}
  $$
  whereas the positive solutions found in
  \cite{Sir_Least_energy_solutions},\cite{LiWe_Ground_states_of},\cite{LiWe_Erratum_ground_states} are
  minimizers of $I$ over the subset $\mathcal{M}_b$ of the Nehari manifold which is given by
  \begin{align*}
    \mathcal{M}_b
    &= \Big\{ (u,v):\; u,v\in H^1(\R^n),\; u,v\neq 0,\; \|u\|^2=\|u\|_{2q}^{2q}+b\|uv\|_q^q,
      \;\|v\|_\omega^2 = \|v\|_{2q}^{2q}+b\|uv\|_q^q\Big\}.
  \end{align*}

  When $b$ is negative, however, the analysis of these constrained minimization problems does not produce any
  fully nontrivial solutions. Indeed, for $b<0$ the minimizers of $I|_{\mathcal{N}_b}$ are given by the
  semitrivial solutions $(\pm u_0,0)$ or $(0,\pm u_0)$ (the latter being possible only for $\omega=1$) where
  $u_0$ is the unique positive function satisfying $-\Delta u_0 + u_0 = u_0^{2q-1}$ in $\R^n$, cf.
  \cite{MaMoPe_Positive_solutions_for},\cite{Kwo_Uniqueness_of_positive}.
  Furthermore it is known that $I|_{\mathcal{M}_b}$ does not admit minimizers in case $b<0$, cf.
  Theorem~1 in \cite{LiWe_Ground_states_of}. Therefore the case of negative coupling parameters $b<0$ has to
  be treated differently. In \cite{Sir_Least_energy_solutions} Sirakov considered the minimization problem
  \begin{equation} \label{Mbrad Gl Defn kappabrad}
    \kappa_b^*  := \inf_{\mathcal{M}_b^*} I\quad\text{where }
    \mathcal{M}_b^* = \big\{ (u,v)\in \mathcal{M}_b:\; u,v\text{ are radially symmetric}\big\}
  \end{equation}
  and he proved the existence of a minimizer of $I|_{\mathcal{M}_b^*}$ for parameter
  values $q=2$ and $n\in\{2,3\}$, cf. Theorem 2 (i). Let us note that the indispensable condition $n\geq 2$ is
  missing in the statement of that theorem.
  \medskip

  The aim of this paper is to generalize Sirakov's result to all space dimensions and to the full range of
  superlinear and subcritical exponents. In Theorem \ref{Mbrad Thm Grenzproblem Mbrad} we
  first investigate the case $n\geq 2$. We show that minimizers $(u_b,v_b)$ of the functional
  $I|_{\mathcal{M}_b^*}$ exist and that, at least up to a subsequence, these minimizers of
  $I|_{\mathcal{M}_b^*}$ converge to a function $(\bar u,\bar v)$ with $\bar u\bar v=0$ and
  \begin{equation} \label{Gl DGL bei -unendlich}
    -\Delta \bar u +  \bar u = \bar u^{2q-1} \quad\text{in }\{\bar u\neq 0\},\qquad
    -\Delta \bar v + \omega^2 \bar v = \bar v^{2q-1} \quad\text{in }\{\bar v\neq 0\}
  \end{equation}
  that solves the optimal partition problem
  \begin{equation} \label{Mbrad Gl Defn kappa-inftyrad}
    \kappa_{-\infty}^*  := \inf \Big\{ \frac{1}{2} \big( \|u\|^2 + \|v\|_\omega^2 \big)
      - \frac{1}{2q} \big(\|u\|_{2q}^{2q} + \|v\|_{2q}^{2q} \big) \; : \; (u,v)\in
      \mathcal{M}_{-\infty}^*\Big\}
  \end{equation}
  where the set $\mathcal{M}_{-\infty}^*$ is defined by
  \begin{equation} \label{Mbrad Gl Defn M-infty}
    \mathcal{M}_{-\infty}^*
    = \Big\{ (u,v):\; u,v\in H^1_r(\R^n),\; u,v\neq 0,\; uv\equiv
    0,\;\|u\|^2=\|u\|_{2q}^{2q}, \;\|v\|_\omega^2 = \|v\|_{2q}^{2q}\Big\}. 
  \end{equation}
  Here, $H^1_r(\R^n)$ denotes the space of radially symmetric functions lying in $H^1(\R^n)$. In particular,
  we find that the supports of $u_b,v_b$ separate as $b\to -\infty$. In general bounded domains
  $\Omega\subset\R^n$ these phenomena have been extensively studied in
  \cite{CoFe_Minimal_coexistence},\cite{CoFe_Global_minimizers_of},\cite{CoTeVe_A_variational_problem},\cite{TaTe_Sign-changin_solutions_of}
  and our Theorem \ref{Mbrad Thm Grenzproblem Mbrad} can be considered as one kind of extension of their
  results.
  \medskip

  In case $n=1$, however, the situation turns out to be different. Since the embedding
  $H^1_r(\R^n)\to L^{2q}(\R^n)$ is not compact for $n=1$ the existence of minimizers of
  $I|_{\mathcal{M}_b^*}$ cannot be proved the same way as in the case $n\geq 2$.
  Therefore we approximate the original problem \eqref{Mbrad Gl Defn kappabrad} by the
  corresponding problem on intervals $B_R=(-R,R)$ for large $R>0$. In Theorem~\ref{Mbrad Thm Kompaktheit
  n=1} we show that for negative coupling parameters $b$ with small absolute value the corresponding
  minimizers converge to a minimizer of $I|_{\mathcal{M}_b^*}$ as $R\to\infty$. For negative $b$ with large absolute value,
  however, we prove in Theorem \ref{Mbrad Thm Nichtexistenz n=1 q leq 2} that solutions of \eqref{Mbrad Gl
  DGL} do not exist at least for exponents $1<q\leq 2$.  \medskip

  Let us present the main results of this paper. The first one deals with the case $n\geq 2$.

  \begin{thm} \label{Mbrad Thm Grenzproblem Mbrad}
     Let $n\geq 2,\omega\geq 1$. Then the following holds:
     \begin{itemize} 
       \item[(i)] The value $\kappa_{-\infty}^*$ is attained at a nonnegative fully nontrivial solution of
       \eqref{Gl DGL bei -unendlich}.
       \item[(ii)] For $b\leq 0$ the value $\kappa_b^*$ is attained at a nonnegative fully nontrivial
       solution of \eqref{Mbrad Gl DGL}.
       \item[(iii)] As $b\to -\infty$ we have $\kappa_b^* \to \kappa_{-\infty}^*$ and every sequence of
       minimizers of ${I|_{\mathcal{M}_b^*}}$ has a subsequence
       $(u_b,v_b)$ such that $|b|^{1/q}u_bv_b\to 0$ in $L^q(\R^n)$ and $(u_b,v_b)\to (\bar u,\bar
       v)$ where the latter function is a fully nontrivial solution of \eqref{Gl DGL bei -unendlich} with
       $\bar u\bar v=0$.
     \end{itemize}
  \end{thm}


  Since the proof of Theorem \ref{Mbrad Thm Grenzproblem Mbrad} makes extensive use of the fact that
  $H^1_r(\R^n)$ embeds compactly into $L^{2q}(\R^n)$ when $n\geq 2$ one has to resort to different
  methods when the space dimension is one.
  In Theorem \ref{Mbrad Thm Nichtexistenz n=1 q leq 2} we show that there is a threshold value
  $b^*(\omega,q)\in [-\infty,0)$ such that $\kappa_b^*$ is attained whenever  $0\geq
  b>b^*(\omega,q)$ whereas $\kappa_b^*$ is not attained for $b<b^*(\omega,q)$. Moreover we find that
  $b^*(\omega,q)$ has the variational characterization
  \begin{equation} \label{Gl Def b*(omega,q)}
    b^*(\omega,q)
    = \inf  \max_{\alpha>0}  \frac{(2+\omega^{\frac{q+1}{q-1}})^{1-q} \|u_0\|^{-2q}\|u_0\|_{2q}^{2q}
         (\|u\|^2+\alpha^2\|v\|_\omega^2)^q- \|u\|_{2q}^{2q}-\alpha^{2q}\|v\|_{2q}^{2q}}{2\alpha^q
        \|uv\|_q^q}
  \end{equation}
  where the infimum is taken over all $u,v\in H^1_r(\R)$ with $uv\neq 0$. As above the
  function $u_0$ appearing in \eqref{Gl Def b*(omega,q)} denotes the positive solution of $-\Delta
  u+u=u^{2q-1}$ in $\R^n$. Our first result dealing with the case $n=1$ reads as follows.

  \begin{thm} \label{Mbrad Thm Kompaktheit n=1}
     Let $n=1,\omega\geq 1$. Then the following holds:
     \begin{itemize} 
       \item[(i)] We have $\kappa_{-\infty}^* = (2+\omega^{\frac{q+1}{q-1}})I(u_0,0)$ and
       $\kappa_{-\infty}^*$ is not attained at any element of $\mathcal{M}_{-\infty}^*$.
       \item[(ii)] If $b<b^*(\omega,q)$ then $\kappa_b^*=\kappa_{-\infty}^*$ and $\kappa_b^*$ is not attained
       at any element of $\mathcal{M}_b^*$.
       \item[(iii)] If $0\geq b>b^*(\omega,q)$ then $\kappa_b^*<\kappa_{-\infty}^*$ and $\kappa_b^*$ is
       attained at a nonnegative fully nontrivial solution of \eqref{Mbrad Gl DGL}.
     \end{itemize}
   \end{thm}

   In view of part (iii) we may prove an explicit sufficient condition for the existence of a fully nontrivial
   solution of \eqref{Mbrad Gl DGL} by estimating the value $b^*(\omega,q)$ from above. To this end we use
   $(u,v)=(u_0,u_0(\omega\cdot))$ as a test function in \eqref{Gl Def b*(omega,q)} which leads to the
   following result.

   \begin{cor} \label{Mbrad Kor hinreichendes Krit Mbrad n=1}
     Let $n=1,\omega\geq 1$. Then for all $b$ satisfying
     \begin{align} \label{Mbrad Gl hinreichendes Krit Mbrad n=1}
 		0\geq b > \max_{\alpha>0} \frac{(2+\omega^{\frac{q+1}{q-1}})^{1-q} (1+\alpha^2\omega)^q-
           1-\alpha^{2q}\omega^{-1}}{2\alpha^q \omega^{-1/2}}
	\end{align}
	the value $\kappa_b^*$ is attained at a nonnegative fully nontrivial solution of \eqref{Mbrad Gl DGL}.
	In particular this is true in case
	 $$
	   \text{(i)}\quad q=2,\; b> -\frac{1}{\omega^{3/2}+\sqrt{2(1+\omega^3)}}
 		\qquad\text{or}\qquad
 		\text{(ii)}\quad \omega=1,\; b> (\frac{2}{3})^{q-1}-1.
	 $$
   \end{cor}

   In order to find necessary conditions for the existence of a minimizer one has to estimate the value
   $b^*(\omega,q)$ from below. For exponents $1<q\leq 2$ we may combine Theorem \ref{Mbrad Thm
   Kompaktheit n=1} (iii) with the following nonexistence result to see that $b^*(\omega,q)$ must be larger
   than or equal to the right hand side in \eqref{Gl Nichtexistenzkriterium n=1}.

  \begin{thm}\label{Mbrad Thm Nichtexistenz n=1 q leq 2}
    Let $n=1, 1<q\leq 2$ and assume
    \begin{align} \label{Gl Nichtexistenzkriterium n=1}
      b< \min_{z>0} \frac{(\omega^2-(q-1)\omega) z^{2q}  - q z^2 - q\omega^3 z^{2q-2}
      - (\omega^2(q-1)-\omega)}{qz^{q+2}+(q-2)(\omega^2+\omega)z^q+q\omega^3 z^{q-2}}.
    \end{align}
    Then the equation \eqref{Mbrad Gl DGL} does not have any fully nontrivial solution. In particular this
    holds in case $q=2,\, b< -\frac{\omega^2+1}{2\omega}$ or $1<q\leq 2,\omega=1,\, b<-1$.
  \end{thm}

  \begin{bem} 
    \begin{itemize}
      \item[(i)] It is worth noticing that Theorem \ref{Mbrad Thm Nichtexistenz n=1 q leq 2} not only
      applies to solutions of minimal energy but to all finite energy solutions.
      \item[(ii)] It would be desirable to know whether a similar nonexistence result is true for exponents
      larger than 2.
      \item[(iii)] From the strong minimum principle for nonnegative supersolutions of elliptic PDEs we know
      that the solutions $(u,v)$ of \eqref{Mbrad Gl DGL} found in Theorem \ref{Mbrad Thm Grenzproblem Mbrad}
      and Theorem \ref{Mbrad Thm Kompaktheit n=1} satisfy $u>0$ and $v>0$ when $q\geq 2$. For $1<q<2$
      we may apply the minimum principle to the function $u+v$ to conclude that $u+v$ is positive. It seems to be
      unclear, however, if both $u$ and $v$ are positive functions in that case.
    \end{itemize}
  \end{bem}

  Finally let us illustrate our main results with two qualitative graphs of the map $b\mapsto \kappa_b^*$ in
  the cases $n\geq 2$ and $n=1,1<q\leq 2$. The monotonicity of this function is referred to at the end of the
  first step in the proof of Theorem \ref{Mbrad Thm Grenzproblem Mbrad}.
  
  ~

   \begin{figure}[h!]
    \centering
	\subfigure[energy levels for $n\geq 2$]
	{
      \begin{tikzpicture}[domain=-30:0, xscale=0.18, yscale=0.5]
    \draw[<-] (-30,0) node[left] {$b$} -- (1,0) node[right,below] {$0$};
    \draw[->] (0,-1) -- (0,9);
    \draw plot  (\x,{ (7*(\x*\x+10)  / (\x*\x+100)) + (13/10)});
    \draw[thin,dashed] (-30,8) -- (0,8) node[right]{$\kappa_{-\infty}^*$};
    \draw[thin,dashed] (-30,2) -- (0,2) node[right]{$\kappa_0^*$};
    \node at (-15,5){$\kappa_b^*$};
     \label{Mbrad fig kappabstar}
    \end{tikzpicture}
  }
  \hfill
 	\subfigure[energy levels for $n=1,1<q\leq 2$]
	{
       \begin{tikzpicture}[domain=-30:0, xscale=0.18, yscale=0.5]
    \draw[<-] (-30,0) node[left] {$b$} -- (1,0) node[right,below] {$0$};
    \draw[->] (0,-1) -- (0,9);
    \draw plot[domain=-15:0]  (\x,{ (7*(\x*\x+10)  / (\x*\x+100)) + (13/10)});
    \draw plot[domain=-30:-15]  (\x,{(7*(225+10)  / (225+100)) + (13/10)});
    \draw[thin,dashed] (-30,{(7*(225+10)  / (225+100)) + (13/10)}) -- (0,{(7*(225+10)  / (225+100)) +
    (13/10)}) node[right]{$\kappa_{-\infty}^*$};
    \draw[thin,dashed] (-30,2) -- (0,2) node[right]{$\kappa_0^*$};
    \node at (-10,3.5){$\kappa_b^*$};
     \label{Mbrad fig kappabstar}
    \end{tikzpicture}
  }
  \end{figure}
 
 ~
 
 \section{Notations and conventions}

  In the following we always assume $n\in\N$ and $1<q<\frac{n}{n-2}$ whenever $n\geq 3$ and $1<q<\infty$
  whenever $n=1$ or $n=2$ so that the Sobolev embedding $H^1_r(\R^n)\to L^{2q}(\R^n)$ exists and
  is compact in case $n\geq 2$. A function $(u,v)$ is called nontrivial if $u\neq 0$ or $v\neq 0$ and it is
  called fully nontrivial in case $u\neq 0$ and $v\neq 0$. The same way $(u,v)$ is nonnegative whenever
  $u\geq 0,v\geq 0$ and it is positive in case $u>0,v>0$. We always consider weak radially
  symmetric solutions of \eqref{Mbrad Gl DGL} and \eqref{Gl DGL bei -unendlich} where it is clear that all solutions of
  \eqref{Mbrad Gl DGL} are twice continuously differentiable on $\R^n$ and smooth in the interior of each
  nodal domain. We use the symbols $\|\cdot\|_r = \|\cdot\|_{L^r(\R^n)}$ to denote the standard Lebesgue norms
  for $1\leq r\leq \infty$ and we set $\|(u,v)\| := \sqrt{\|u\|^2+\|v\|_\omega^2}$ for $u,v\in H^1_r(\R^n)$
  where
  \begin{gather}  \label{Gl Defn norms}
    \|u\|
    := \Big( \int_{\R^n} |\nabla u|^2+u^2\dx \Big)^{1/2},\quad
    \|v\|_\omega
    := \Big( \int_{\R^n} |\nabla v|^2+\omega^2 v^2\dx \Big)^{1/2}.
  \end{gather}
  From the definition of $I$ we get
  \begin{equation} \label{Gl für (u,v) in Nehari}
    I(u,v) = \frac{q-1}{2q} (\|u\|^2+\|v\|_\omega^2) \qquad\text{for all }(u,v)\in\mathcal{N}_b
  \end{equation}
  and in particular for all elements of $\mathcal{M}_b$ or $\mathcal{M}_b^*$. The function $u_0\in
  H^1_r(\R^n)$ is defined as above and for notational convenience we put $c_0:=I(u_0,0)$.
  We set $v_0:= \omega^{1/(q-1)}u_0(\omega\cdot)$ so that $v_0$ is the unique positive
  solution of $-\Delta v + \omega^2 v = v^{2q-1}$ in $\R^n$. A short calculation shows
  $$
    I(0,v_0)=\omega^{\frac{2q-n(q-1)}{q-1}} c_0.
  $$
  We will use the facts that the functions $u_0,v_0$ are minimizers of the functionals
  $\frac{\|u\|_{~~}}{\|u\|_{2q}}, \frac{\|v\|_{\omega~}}{\|v\|_{2q}}$, respectively and that all
  minimizers of these functionals are translates of $u_0,v_0$. Moreover, we use that $(u_0,0)$ is a minimizer
  of the functional $I|_{\mathcal{N}_b}$ when $b<0$.



  \section{Proof of Theorem \ref{Mbrad Thm Grenzproblem Mbrad}}

  Throughout this section except for the first step we assume $n\geq 2$ according to the assumptions of
  Theorem~\ref{Mbrad Thm Grenzproblem Mbrad}. Its proof is given in four steps. First we prove
  variational characterizations for the values $\kappa_b^*,\kappa_{-\infty}^*$ which turn out to be more
  convenient than the original ones given by \eqref{Mbrad Gl Defn kappabrad} and \eqref{Mbrad Gl Defn
  kappa-inftyrad}. In the second step we use these characterizations to prove that minimizers of the
  functionals $I|_{\mathcal{M}_b^*}$ and $I|_{\mathcal{M}_{-\infty}^*}$ exist. In the third step we show that
  minimizers satisfy the corresponding Euler-Lagrange equation \eqref{Mbrad Gl DGL} or \eqref{Gl DGL bei
  -unendlich} so that the assertions (i) and (ii) of the theorem are proved. Finally we show part (iii) of the theorem.
  \medskip
  \medskip

  {\it Step 1: A more convenient variational characterization for $\kappa_b^*,\kappa_{-\infty}^*$ \;($n\geq 1$)}
  \medskip

  For $s,t>0$ and $u,v\in H^1_r(\R^n)$ with $u,v\neq 0$ one can check that $(su,tv)\in\mathcal{M}_b^*$
  is equivalent to $(s,t)$ being a critical point of the function $\beta_{u,v}$ defined
  on $\R_{>0}\times\R_{>0}$ and given by
  $$
    \beta_{u,v}(\tilde s,\tilde t)
    := I(\tilde su,\tilde tv)
    = \frac{\tilde s^2}{2}\|u\|^2 +\frac{\tilde t}{2}\|v\|_\omega^2
     - \frac{\tilde s^{2q}}{2q} \|u\|_{2q}^{2q} - \frac{\tilde t^{2q}}{2q} \|v\|_{2q}^{2q} - \frac{b\tilde
     s^q\tilde t^q}{q} \|uv\|_q^q.
  $$
  A necessary and sufficient condition for such a critical point to exist is given by
  \begin{equation} \label{Mbrad Gl Skalierungsbedingung}
     \|u\|_{2q}^q\|v\|_{2q}^q + b\|uv\|_q^q >0.
  \end{equation}
  Indeed, in this case the functional $-\beta_{u,v}$ is coercive so that $\beta_{u,v}$ has a global maximum.
  Moreover one can show that the Hessian of the function $(\tilde s, \tilde t)\mapsto \beta_{u,v}(\tilde
  s^{1/2q},\tilde t^{1/2q})$ is positive definite on $\R_{>0}\times\R_{>0}$ so that
  the maximum is strict and no other critical point can exist. On the other hand a short calculation shows
  that \eqref{Mbrad Gl Skalierungsbedingung} is also a necessary condition for the existence of a critical
  point. From
  \begin{align*}
    \max_{s,t>0} \beta_{u,v}(s,t)
    &= \max_{s,t>0} I(su,tv) \\
    &= \max_{\alpha>0} \max_{s>0} I(su,s\alpha v) \\
    &= \max_{\alpha>0} \max_{s>0}
       \frac{s^2}{2}(\|u\|^2 + \alpha^2 \|v\|_\omega^2)
       - \frac{s^{2q}}{2q} (\|u\|_{2q}^{2q} + \alpha^{2q} \|v\|_{2q}^{2q} +2b\alpha^q \|uv\|_q^q ).\\
    &=  \frac{q-1}{2q} \Big( \max_{\alpha>0}\frac{(\|u\|^2+\alpha^2\|v\|_\omega^2)^q
    }{\|u\|_{2q}^{2q}+\alpha^{2q}\|v\|_{2q}^{2q} + 2b\alpha^q \|uv\|_q^q} \Big)^{\frac{1}{q-1}}
  \end{align*}
  we obtain the following variational characterization for $\kappa_b^*$:
   \begin{align} \label{Mbrad Gl Charakerisierung kappabrad}
     \begin{aligned}
     &\kappa_b^*
     = \inf \Big\{ \frac{q-1}{2q} \hat J(u,v)^{\frac{1}{q-1}} : u,v\in H^1_r(\R^n),\, (u,v)
     \text{ satisfies }\eqref{Mbrad Gl Skalierungsbedingung}\Big\} \\
     &\text{where}\quad \hat J(u,v) =  \max_{\alpha>0}\frac{(\|u\|^2+\alpha^2\|v\|_\omega^2)^q
    }{\|u\|_{2q}^{2q}+\alpha^{2q}\|v\|_{2q}^{2q} + 2b\alpha^q \|uv\|_q^q}.
    \end{aligned}
   \end{align}
   Moreover if $(u,v)$ satisfies \eqref{Mbrad Gl Skalierungsbedingung} and minimizes $\hat J$ then $(su,tv)$
   is a minimizer of $I|_{\mathcal{M}_b^*}$ provided $(s,t)$ is the unique maximizer of $\beta_{u,v}$.
   Similarly, one can show
   \begin{align} \label{Gl Grenzproblem bei -unendlich}
     \begin{aligned}
     &\kappa_{-\infty}^*
     =  \inf \Big\{ \frac{q-1}{2q} \bar J(u,v)^{\frac{1}{q-1}} : u,v\in H^1_r(\R^n), u,v\neq 0,
     uv= 0 \Big\} \\
     &\text{where}\quad  \bar J(u,v)
     = \max_{\alpha>0} \frac{(\|u\|^2+\alpha^2\|v\|_\omega^2)^q}{\|u\|_{2q}^{2q}+\alpha^{2q}\|v\|_{2q}^{2q}}
     = \Big(\big(\frac{\|u\|_{~~}}{\|u\|_{2q}} \big)^{\frac{2q}{q-1}} +
       \big(\frac{\|v\|_{\omega~}}{\|v\|_{2q}} \big)^{\frac{2q}{q-1}} \Big)^{q-1}.
     \end{aligned}
  \end{align}
   Since the constraint $uv= 0$ is more restrictive than \eqref{Mbrad Gl Skalierungsbedingung} we obtain
   the inequality
   \begin{equation} \label{Gl Ungleichung kappabrad kappa-unendlichrad}
     \kappa_b^* \leq \kappa_{-\infty}^* \qquad( b\leq 0).
   \end{equation}
   Moreover, from \eqref{Mbrad Gl Charakerisierung kappabrad} it follows that the map $b\mapsto \kappa_b^*$ is
   nonincreasing.
   \medskip\medskip

   \noindent{\it Step 2: Existence of nonnegative minimizers}
   \medskip

   We prove that both $\kappa_b^*$ and $\kappa_{-\infty}^*$ are attained at nonnegative elements of
   $\mathcal{M}_b^*,\mathcal{M}_{-\infty}^*$, respectively. By the first step it suffices to show that the
   functionals $\hat J,\bar J$ defined in \eqref{Mbrad Gl Charakerisierung kappabrad},\eqref{Gl
   Grenzproblem bei -unendlich} admit fully nontrivial nonnegative minimizers. Since the reasonings for $\hat
   J$ and $\bar J$ are almost identical, we only give the proof for $\hat J$.
    \medskip

   Let $(u_j,v_j)$ be a minimizing sequence for $\hat J$ satisfying \eqref{Mbrad Gl Skalierungsbedingung}.
   Since $\hat J(u_j,v_j)=\hat J(s|u_j|,t|v_j|)$ for all $s,t>0$ we may assume ${u_j,v_j\geq 0}$ as well as
   $\|u_j\|_{2q}=\|v_j\|_{2q}=1$. Then $(u_j,v_j)$ is bounded and there is a subsequence $(u_j,v_j)$ that, due
   to the compactness of the embedding $H^1_r(\R^n)\to L^{2q}(\R^n)$, converges weakly, almost everywhere and
   in $L^{2q}(\R^n)\times L^{2q}(\R^n)$ to some nonnegative function $(u,v)$.
   This entails ${\|u\|_{2q}=\|v\|_{2q}=1}$ as well as $ u, v\geq
   0$. Furthermore, $(u,v)$ satisfies \eqref{Mbrad Gl Skalierungsbedingung} because otherwise $\hat
   J(u_j,v_j)$ would tend to infinity as $j\to\infty$ contradicting its property of a minimizing sequence.
   Hence, for all $\alpha>0$ we have
   \begin{align*}
      \frac{(\|u\|^2+\alpha^2\|v\|_\omega^2)^q
      }{\|u\|_{2q}^{2q}+\alpha^{2q}\|v\|_{2q}^{2q} + 2b\alpha^q \|uv\|_q^q}
     &\leq \liminf_{j\to\infty}  \frac{(\|u_j\|^2+\alpha^2\|v_j\|_\omega^2)^q
     }{\|u_j\|_{2q}^{2q}+\alpha^{2q}\|v_j\|_{2q}^{2q} + 2b\alpha^q \|u_jv_j\|_q^q}   \\
     &\leq \liminf_{j\to\infty}  \max_{\beta>0} \frac{(\|u_j\|^2+\beta^2\|v_j\|_\omega^2)^q
     }{\|u_j\|_{2q}^{2q}+\beta^{2q}\|v_j\|_{2q}^{2q} + 2b\beta^q \|u_jv_j\|_q^q}.
    \end{align*}
    Using \eqref{Mbrad Gl Charakerisierung kappabrad} we find $\hat J(u,v)\leq \liminf_{j\to\infty} \hat
    J(u_j,v_j)$ so that $(u,v)$ is a minimizer of~$\hat J$.
    \medskip
     \medskip

   \noindent{\it Step 3: The solution property}
   \medskip

   We prove the following two statements:
   \begin{itemize}\vspace{-0.2\baselineskip}
     \item[(i)] In case $b\leq 0$ every minimizer of $I|_{\mathcal{M}_b^*}$ is a solution of
     \eqref{Mbrad Gl DGL}.
     \item[(ii)] Every minimizer of $I|_{\mathcal{M}_{-\infty}^*}$ is a solution of
     \eqref{Gl DGL bei -unendlich}.
   \end{itemize}
   Let us show assertion (i) first. For $u,v\in H^1_r(\R^n)$ with $u,v\neq 0$ set
   $$
     H_1(u,v):= \|u\|^2- \|u\|_{2q}^{2q}-b\|uv\|_q^q,\quad\text{and}\quad
     H_2(u,v) := \|v\|_\omega^2-\|v\|_{2q}^{2q}-b\|uv\|_q^q.
   $$
   so that $(u,v)\in \mathcal{M}_b^*$ if and only if $H_1(u,v)=H_2(u,v)=0$. Now, if $(u,v)\in
   \mathcal{M}_b^*$ is a minimizer of $I|_{\mathcal{M}_b^*}$ then $H_1(u,v)=H_2(u,v)=0$ implies
    \begin{align*}
      H_1'(u,v)[u,0]
      &= 2\|u\|^2- 2q\|u\|_{2q}^{2q}- qb\|uv\|_q^q
       = (2-2q)\|u\|^2 + bq\|uv\|_q^q
       < 0, \\
      H_2'(u,v)[0,v]
      &= 2\|v\|_\omega^2-2q\|v\|_{2q}^{2q}-q b\|uv\|_q^q
        = (2-2q)\|v\|_\omega^2 + bq\|uv\|_q^q
       < 0
    \end{align*}
    so that Lagrange's multiplier rule shows that there are $L_1,L_2\in\R$ such that
    \begin{equation} \label{Gl ELequation}
      I'(u,v)+L_1 H_1'(u,v)+L_2 H_2'(u,v)=0.
    \end{equation}
    It suffices to show $L_1=L_2=0$.
    \medskip

    Using $(u,0),(0,v)$ as test functions in \eqref{Gl ELequation} we find
    $\skp{I'(u,v)}{(u,0)}=H_1(u,v)=0$ and $\skp{I'(u,v)}{(0,v)}=H_2(u,v)=0$ and thus
    \begin{align*}
       0 &= \Big((2-2q)\|u\|_{2q}^{2q}+(2-q)b\|uv\|_q^q\Big) L_1 - qb\|uv\|_q^q L_2,  \\
       0 &= \Big((2-2q)\|v\|_{2q}^{2q}+(2-q)b\|uv\|_q^q\Big) L_2 - qb\|uv\|_q^q L_1.
    \end{align*}
    Assume $(L_1,L_2)\neq (0,0)$. Then $\|uv\|_q^q>0$ and the determinant of this system vanishes. We
    therefore get
	\begin{align*}
	  0
	  &=  \Big((2-2q)\|u\|_{2q}^{2q}+(2-q)b\|uv\|_q^q\Big)\cdot \Big((2-2q)\|v\|_{2q}^{2q}+(2-q)b\|uv\|_q^q\Big)
	  - (qb\|uv\|_q^q)^2 \\
	  &= 4(1-q) \Big((b\|uv\|_q^q)^2  -
	  \frac{q-2}{2} b\|uv\|_q^q(\|u\|_{2q}^{2q}+\|v\|_{2q}^{2q}) - (q-1)\|u\|_{2q}^{2q}\|v\|_{2q}^{2q} \Big).
	\end{align*}
	Solving for $b\|uv\|_q^q<0$ gives
	\begin{align*}
	  4b\|uv\|_q^q
	  &= (q-2)(\|u\|_{2q}^{2q}+\|v\|_{2q}^{2q})-\sqrt{(q-2)^2(\|u\|_{2q}^{2q}+\|v\|_{2q}^{2q})^2+
	  16(q-1)\|u\|_{2q}^{2q}\|v\|_{2q}^{2q}}.
	\end{align*}
	Let now $A,B>0$ be given by $\|u\|_{2q}^{2q} = A\cdot |b|\|uv\|_q^q,\|v\|_{2q}^{2q}=B\cdot
	|b|\|uv\|_q^q$.
	Then
	\begin{equation} \label{Mbrad Gl AB}
      -4 = (q-2)(A+B) - \sqrt{(q-2)^2(A+B)^2+16(q-1)AB}.
    \end{equation}
	where $A,B$ are larger than 1 because of
    \begin{equation*}
      \|u\|_{2q}^{2q}-|b|\|uv\|_q^q = \|u\|^2 > 0,\qquad
      \|v\|_{2q}^{2q}-|b|\|uv\|_q^q = \|v\|_\omega^2 > 0.
    \end{equation*}
    Solving \eqref{Mbrad Gl AB} for $B$ we obtain
    $$
      B = \frac{2+(q-2)A}{2(q-1)A-(q-2)}
    $$
    so that $A>1$ implies $B<1$, a contradiction. Hence, the assumption was false, i.e. ${I'(u,v)=0}$.
    \medskip

    Now consider (ii). Let $(\bar u,\bar v)\in\mathcal{M}_{-\infty}^*$ be a minimizer of
    the functional $I|_{\mathcal{M}_{-\infty}^*}$. Due to the one-dimensional Sobolev embedding we may
    choose $\bar u,\bar v$ to be a continuous function on $\R^n\sm\{0\}$ so that the sets $\{\bar u\neq
    0\},\{\bar v\neq 0\}$ are open. According to the first step we have $\bar J(\bar
    u,\bar v)\leq \bar J(\bar u+\varphi,\bar v+\psi)$ for all test functions $\varphi,\psi$ with
    $\supp(\varphi)~\subset~\supp(u)$ and $\supp(\psi)\subset \supp(v)$. In view of the second formula for
    $\bar J$ in \eqref{Gl Grenzproblem bei -unendlich} we find that $(\bar u,\bar v)$ solves~\eqref{Gl DGL bei
    -unendlich}. \qed

    \begin{bem}
      The above reasoning shows that all critical points and not only minimizers of
      $I|_{\mathcal{M}_b^*}$ or $I|_{\mathcal{M}_{-\infty}^*}$ satisfy the corresponding
      Euler-Lagrange equation.
    \end{bem}
    \medskip

    \noindent{\it Step 4: Convergence to a fully nontrivial solution of \eqref{Gl DGL bei -unendlich} as $b\to
    -\infty$}
    \medskip

  Let $(b_j)$ be a sequence such that $b_j\to -\infty$ and let $(u_j,v_j)\in\mathcal{M}_{b_j}^*$ be a sequence
  of nonnegative fully nontrivial solutions of \eqref{Mbrad Gl DGL} given by the second step, in
  particular ${I(u_j,v_j)=\kappa_{b_j}^*}$. Then $(u_j,v_j)$ is bounded and there is a subsequence
  $(u_j,v_j)$ that, due to the compactness of the embedding $H^1_r(\R^n)\to L^{2q}(\R^n)$, converges
  weakly, almost everywhere and in ${L^{2q}(\R^n)\times L^{2q}(\R^n)}$ to some nonnegative function $(\bar
  u,\bar v)$. From Sobolev's inequality we get
  \begin{align*}
     \|u_j\|^2
        &= \|u_j\|_{2q}^{2q}+b_j\|u_jv_j\|_q^q
        \leq \|u_j\|_{2q}^{2q}
        \leq C \|u_j\|^{2q}, \\
        \|v_j\|_\omega^2
        &= \|v_j\|_{2q}^{2q}+b_j\|u_jv_j\|_q^q
        \leq \|v_j\|_{2q}^{2q}
        \leq C \|v_j\|_\omega^{2q}
  \end{align*}
  and thus $\|u_j\|_{2q},\|v_j\|_{2q}\geq c>0$ where $c,C$ are positive numbers which do not depend
  on~$j$. It follows $\|\bar u\|_{2q},\|\bar v\|_{2q}\geq c$ and thus $\bar u,\bar v\neq 0$. In addition we
  find
  \begin{equation}  \label{Gl II Grenzproblem Mbrad}
    \|\bar u\|^2\leq \|\bar u\|_{2q}^{2q},\qquad
    \|\bar v\|_\omega^2\leq \|\bar v\|_{2q}^{2q}.
  \end{equation}
  Since the sequence $(u_j,v_j)$ is bounded we get $\bar u\bar v\equiv 0$ from
  $$
         \|\bar u\bar v\|_q^q
         = \lim_{j\to\infty} \|u_jv_j\|_q^q
         = \lim_{j\to\infty} (\|u_j\|_{2q}^{2q}- \|u_j\|^2)|b_j|^{-1}
         \leq \liminf_{j\to\infty} C\cdot |b_j|^{-1}
         = 0.
  $$
  Furthermore, from \eqref{Gl Ungleichung kappabrad kappa-unendlichrad} we obtain $\kappa_{b_j}^*\leq
  \kappa_{-\infty}^*$ so that \eqref{Gl für (u,v) in Nehari} implies
  \begin{align*}
     \frac{q-1}{2q}(\| \bar u\|^2+\|\bar v\|_\omega^2)
       &\leq \frac{q-1}{2q} \liminf_{j\to\infty} (\|u_j\|^2+\|v_j\|_\omega^2) \\
       &= \liminf_{j\to\infty} \kappa_{b_j}^*   \\
       &\leq \limsup_{j\to\infty} \kappa_{b_j}^* \\
       &\leq \kappa_{-\infty}^*  \\
       &\leq  \frac{q-1}{2q} \Big(  \big(\frac{\|\bar  u\|_{~~}}{\| \bar u\|_{2q}}\big)^{\frac{2q}{q-1}}
         + \big(\frac{\|\bar v\|_\omega~}{\|\bar v\|_{2q}}\big)^{\frac{2q}{q-1}} \Big)  \\
       &\leq \frac{q-1}{2q} (\|\bar u\|^2 + \|\bar  v\|_\omega^2)
  \end{align*}
  where we used  \eqref{Gl Grenzproblem bei -unendlich} and \eqref{Gl II Grenzproblem Mbrad} in the last two
  inequalities. Hence, equality occurs in each line and thus $\|\bar u\|^2= \|\bar u\|_{2q}^{2q}$, $\|\bar
  v\|_\omega^2= \|\bar v\|_{2q}^{2q}$ as well as $\kappa_{b_j}^* \to \kappa_{-\infty}^*$, $(u_j,v_j)\to (\bar
  u,\bar v)$ as $b\to -\infty$. This entails $(\bar u,\bar v)\in\mathcal{M}_{-\infty}^*$ and $I(\bar
  u,\bar v)=\kappa_{-\infty}^*$ so that $(\bar u,\bar v)$ is a minimizer of
  $I|_{\mathcal{M}_{-\infty}^*}$ and thus a fully nontrivial nonnegative solution of \eqref{Gl DGL bei
  -unendlich}. Finally, the assertion follows from
  $$
    \limsup_{j\to\infty} |b_j|\|u_jv_j\|_q^q
    = \limsup_{j\to\infty} \|u_j\|_{2q}^{2q}-\|u_j\|^2
     = \|\bar u\|_{2q}^{2q}-\|\bar u\|^2
     = 0.
  $$

  \section{Proof of Theorem \ref{Mbrad Thm Kompaktheit n=1} and Corollary \ref{Mbrad Kor hinreichendes Krit
  Mbrad n=1}}

  {\it  Proof of (i)}
  \medskip

  First we show $\kappa_{-\infty}^*\geq (2+\omega^{\frac{q+1}{q-1}})c_0$ and that no element of
  $\mathcal{M}_{-\infty}^*$ attains this value. Let $(u,v)\in\mathcal{M}_{-\infty}^*$ and in particular
  $u(0)v(0)=0$. We first assume $u(0)=0$. Then the nontrivial functions $u_l:= u\cdot 1_{(-\infty,0)},u_r:=
  u\cdot 1_{(0,\infty)}$ lie in $H^1(\R)$, they have disjoint support and  satisfy $u_r(r)=u_l(-r)$ due
  to $u\in H^1_r(\R)$. In particular from $\|u\|^2=\|u\|_{2q}^{2q}$ we infer
  $$
    \|u_l\|^2 = \|u_r\|^2
    = \frac{1}{2}\|u\|^2
    = \frac{1}{2}\|u\|_{2q}^{2q}
    = \|u_l\|_{2q}^{2q} = \|u_r\|_{2q}^{2q}.
  $$
  This implies $(u_r,0),(u_l,0),(0,v)\in\mathcal{N}_b$ and using \eqref{Gl für (u,v) in Nehari} as well as
  $uv\equiv 0$ we obtain
  \begin{align*}
    I(u,v)
    &= I(u_l,0) + I(u_r,0) + I(0,v) \\
    &= \frac{q-1}{2q}\cdot \big( \|u_l\|^2 + \|u_r\|^2 + \|v\|_\omega^2 \big) \\
    &= \frac{q-1}{2q}\cdot \Big( \big( \frac{\|u_l\|_{~~}}{\|u_l\|_{2q}}\big)^{\frac{2q}{q-1}} +
      \big( \frac{\|u_r\|_{~~}}{\|u_r\|_{2q}}\big)^{\frac{2q}{q-1}}
      + \big( \frac{\|v\|_{\omega~}}{\|v\|_{2q}}\big)^{\frac{2q}{q-1}} \Big).
    \intertext{Since the functions $u_0,v_0$ minimize the quotiens
    $\frac{\|u\|_{~~}}{\|u\|_{2q}},\frac{\|v\|_{\omega~}}{\|v\|_{2q}}$ we get}
    I(u,v)
    &\geq \frac{q-1}{2q}\cdot \Big( 2\cdot \big( \frac{\|u_0\|_{~~}}{\|u_0\|_{2q}}\big)^{\frac{2q}{q-1}}
      + \big( \frac{\|v_0\|_{\omega~}}{\|v_0\|_{2q}}\big)^{\frac{2q}{q-1}} \Big) \\
    &= \frac{q-1}{2q}\cdot \big(2\|u_0\|^2+\|v_0\|_\omega^2 \big) \\
    &= 2I(u_0,0) + I(0,v_0) \\
    &= (2+\omega^{\frac{q+1}{q-1}})c_0.
  \end{align*}
  Analogously the assumption $v(0)=0$ leads to
  $$
    I(u,v)
    \geq (1+2\omega^{\frac{q+1}{q-1}})c_0
    \geq (2+\omega^{\frac{q+1}{q-1}})c_0.
  $$
  We therefore get $\kappa_{-\infty}^* \geq (2+\omega^{\frac{q+1}{q-1}})c_0$. Moreover we find that
  $\kappa_{-\infty}^*$ is not attained at any element of $\mathcal{M}_{-\infty}^*$ because in case $u(0)=0$
  this would lead to the conclusion that $u_r,u_l$ are translates of $u_0$ which is impossible because of
  $\supp(u_r)\cap \supp(u_l)=\emptyset$. A similar reasoning shows that no element $(u,v)$ of $\mathcal{M}_{-\infty}^*$
  with $v(0)=0$ can have energy $(2+\omega^{\frac{q+1}{q-1}})c_0$.
  \medskip

  Now let us prove the opposite inequality. To this end let $\chi_k:=\chi(k^{-1}\cdot)$ denote a suitable
  radially symmetric cut-off function with $\chi\equiv 1$ in $[-1,1]$ and $\chi\equiv 0$ outside of $(-2,2)$
  then the sequence
  $$
    (u_k,v_k):= \Big( (u_0\chi_k)(2k+\cdot)+(u_0\chi_k)(-2k+\cdot),v_0\chi_k \Big)
  $$
  lies in $\mathcal{M}_{-\infty}^*$ and
  $$
    \lim_{k\to\infty} I(u_k,v_k)
    = \lim_{k\to\infty} \big( 2  I(u_0\chi_k,0) + I(0,v_0\chi_k) \big)
    =  2 I(u_0,0) + I(0,v_0)
    = (2+\omega^{\frac{q+1}{q-1}})c_0
  $$
  which proves $\kappa_{-\infty}^* \leq (2+\omega^{\frac{q+1}{q-1}})c_0$. Hence, we obtain
  $$
    \kappa_{-\infty}^* = (2+\omega^{\frac{q+1}{q-1}})c_0.
  $$
  \medskip

  {\it Proof of (ii)}
  \medskip

  First we prove that $b<b^*(\omega,q)$ implies $\kappa_b^*=\kappa_{-\infty}^*$ and that $0\geq
  b>b^*(\omega,q)$ implies ${\kappa_b^*<\kappa_{-\infty}^*}$. From (i) and the variational characterization
  for $\kappa_b^*$ given by \eqref{Mbrad Gl Charakerisierung kappabrad} we get $\kappa_b^*<\kappa_{-\infty}^*$ if
  and only if there are functions $u,v\in H^1_r(\R^n)$ with $u,v\neq 0$ and $\|u\|_{2q}^q\|v\|_{2q}^q>|b|\|uv\|_q^q$ that satisfy
  $$
     \max_{\alpha>0}
     \frac{(\|u\|^2+\alpha^2\|v\|_\omega^2)^q}{\|u\|_{2q}^{2q}+\alpha^{2q}\|v\|_{2q}^{2q}+2b\alpha^q\|uv\|_q^q}
     < \Big(\frac{2q}{q-1}\cdot (2+\omega^{\frac{q+1}{q-1}})c_0\Big)^{q-1}
     = (2+\omega^{\frac{q+1}{q-1}})^{q-1} \frac{\|u_0\|^{2q}}{\|u_0\|_{2q}^{2q}}.
  $$
 In this case (i) and \eqref{Gl Grenzproblem bei -unendlich} implies $uv\neq 0$ and thus $b>b^*(\omega,q)$
 after some rearrangements of the above inequality. A short argument shows that
 $0\geq b>b^*(\omega,q)$ implies $\kappa_b^*<\kappa_{-\infty}^*$.
 \medskip

  Let $b<b^*(\omega,q)$ and assume that $\kappa_b^*=\kappa_{-\infty}^*$ is attained at some function $(u,v)$
  satisfying the condition \eqref{Mbrad Gl Skalierungsbedingung}. Then $uv\neq 0$ since $\kappa_{-\infty}^*$
  is not attained, see (i). Choose $\eps>0$ such that $b+\eps<b^*(\omega,q)$ so that $ \kappa_{b+\eps}^* =
  \kappa_{-\infty}^*$. Then $(u,v)$ satisfies \eqref{Mbrad Gl Skalierungsbedingung} for $b+\eps$
  instead of $b$ and we get
  \begin{align*}
    \kappa_b^*
    &= \frac{q-1}{2q} \Big( \max_{\alpha>0} \frac{(\|u\|^2 +
    \alpha^2\|v\|_\omega^2)^q}{\|u\|_{2q}^{2q} + \alpha^{2q}\|v\|_{2q}^{2q} + 2b\alpha^q\|uv\|_q^q}
    \Big)^{\frac{1}{q-1}}  \\
    &>  \frac{q-1}{2q} \Big( \max_{\alpha>0} \frac{(\|u\|^2 +
    \alpha^2\|v\|_\omega^2)^q}{\|u\|_{2q}^{2q} + \alpha^{2q}\|v\|_{2q}^{2q} + 2(b+\eps)\alpha^q\|uv\|_q^q}
    \Big)^{\frac{1}{q-1}}  \\
    &\geq \kappa_{b+\eps}^* \\
    &= \kappa_{-\infty}^*
  \end{align*}
  which contradicts $\kappa_b^*=\kappa_{-\infty}^*$. Hence, $\kappa_b^*$ is not attained for
  $b<b^*(\omega,q)$.
  \medskip
  \medskip

  {\it Proof of (iii)}

   In order to prove (iii) we suppose $0\geq b>b^*(\omega,q)$. From the first statement in the proof of (ii)
   it follows that this implies
   \begin{equation}\label{Gl Ungl kappab kappa-infty}
     \kappa_b^*<\kappa_{-\infty}^*=(2+\omega^{\frac{q+1}{q-1}})c_0.
   \end{equation}
   For these values of $b$ let us investigate the behaviour of a special minimizing sequence for
   the functional $I|_{\mathcal{M}_b^*}$. We consider the corresponding problem on balls $B_R=(-R,R)$ where
   $R$ will be sent to infinity. We set $H^1_{0,r}(B_R) := \{ u\in H_0^1(B_R) : u \text{ is radially
   symmetric} \}$. All solutions $(u,v)\in H^1_{0,r}(B_R)\times H^1_{0,r}(B_R)$ of the boundary value problem
  \begin{align} \label{Mbrad Gl DGL uR vR}
    \begin{gathered}
       - u'' + ~~~ u = |u|^{2q-2}u + b|u|^{q-2}u|v|^q\quad\text{in }B_R, \\
       - v'' + \omega^2 v = |v|^{2q-2}v + b|v|^{q-2}v|u|^q\quad\text{in }B_R, \\
       u(-R)=u(R) = 0,\quad  v(-R)=v(R)=0
    \end{gathered}
  \end{align}
  satisfy
  \begin{align}
    \int_{B_R} |u'|^2+ ~~~u^2\dx &= \int_{B_R} |u|^{2q}+b|u|^q|v|^q\dx,
    \label{Mbrad Gl MbstarR Bedingung 1}  \\
     \int_{B_R} |v'|^2+\omega^2 v^2\dx &= \int_{B_R} |v|^{2q}+b|u|^q|v|^q\dx. \label{Mbrad Gl MbstarR
     Bedingung 2}
  \end{align}
  Following the approach of the last section we define
  \begin{align*}
    \mathcal{M}_b^*(R)
    &:= \Big\{ (u,v)\in H^1_{0,r}(B_R)\times H^1_{0,r}(B_R) : u,v\neq 0,\, (u,v) \text{ satisfies }
    \eqref{Mbrad Gl MbstarR Bedingung 1},\eqref{Mbrad Gl MbstarR Bedingung 2} \Big\}.
  \end{align*}
  As before one can show that $\inf I|_{\mathcal{M}_b^*(R)}$ admits a variational
  characterization given by
  \begin{align} \label{Mbrad Gl Charakterisierung kappabradR}
    \kappa_b^*(R)
    := \inf_{\mathcal{M}_b^*(R)} I
    = \inf \Big\{ \frac{q-1}{2q} \hat J(u,v)^{\frac{1}{q-1}} : u,v\in
    H^1_{0,r}(B_R),\,(u,v) \text{ satisfies }\eqref{Mbrad Gl Skalierungsbedingung}\Big\}.
  \end{align}
  Using the compactness of the embedding $H^1_{0,r}(B_R)\to L^{2q}(B_R)$ for all $R>0$ we obtain the following
  result:
  \medskip

   \begin{prop}\label{Mbrad Prop Mbrad-Minimierung n=1 Approx durch Kugeln}
     Let $n=1$. For all $b\leq 0$ the value $\kappa_b^*(R)$ is attained at a fully nontrivial nonnegative
     solution $(u_R,v_R)\in H_{0,r}^1(B_R)\times H_{0,r}^1(B_R)$ of \eqref{Mbrad Gl DGL uR vR}. Moreover,
     as $R\to\infty$ we have $\kappa_b^*(R)\to \kappa_b^*$.
   \end{prop}
   \begin{proof}
     The existence of a fully nontrivial nonnegative minimizer of $I|_{\mathcal{M}_b^*(R)}$ can be shown as
     in the second step in the proof of Theorem \ref{Mbrad Thm Grenzproblem Mbrad}. From the inclusion
     $\mathcal{M}_b^*(R)\subset \mathcal{M}_b^*$ it follows
     \begin{equation} \label{Mbrad Gl Monotonie kappabradR}
       \kappa_b^*(R)\geq \kappa_b^*.
     \end{equation}
     In order to show $\kappa_b^*(R)\to \kappa_b^*$ as $R\to\infty$ we choose a cut-off function $\chi$ with
     $\chi(x)=1$ for $|x|\leq \frac{1}{2}$ and $\chi(x)=0$ für $|x|\geq 1$, set $\chi_R(x):=\chi(R^{-1}x)$.
     Then for all $u,v\in H^1_r(\R)$ we have $u\chi_R,v\chi_R\in H^1_{0,r}(B_R)$. Moreover, if in addition
     $(u,v)$ satisfies \eqref{Mbrad Gl Skalierungsbedingung} so does $(u\chi_R,v\chi_R)$ for
     sufficiently large $R>0$ and we get from \eqref{Mbrad Gl Charakterisierung kappabradR} the inequality
     $$
       \limsup_{R\to\infty} \kappa_b^*(R)
       \leq \frac{q-1}{2q} \limsup_{R\to\infty} \hat J(u\chi_R,v\chi_R)^{\frac{1}{q-1}}
       = \frac{q-1}{2q} \hat J(u,v)^{\frac{1}{q-1}}.
     $$
     Since this holds for all $u,v\in H^1_r(\R)$ satisfying \eqref{Mbrad Gl Skalierungsbedingung} we obtain
     from \eqref{Mbrad Gl Charakerisierung  kappabrad} the estimate
     \begin{equation} \label{Mbrad Gl Asymptotik kappabradR}
       \limsup_{R\to\infty} \kappa_b^*(R) \leq \kappa_b^*.
     \end{equation}
     The inequalities \eqref{Mbrad Gl Monotonie kappabradR} and \eqref{Mbrad Gl Asymptotik kappabradR}
     show $\kappa_b^*(R)\to \kappa_b^*$ as $R\to\infty$.
   \end{proof}
   \medskip

  Let now $(u_k,v_k):=(u_{R_k},v_{R_k})$ be the sequence of solutions on $(-R_k,R_k)$ given by
  Proposition~\ref{Mbrad Prop Mbrad-Minimierung n=1 Approx durch Kugeln} where $(R_k)$ is a fixed positive
  sequence going off to infinity as $k\to\infty$. Then $(u_k,v_k)$ lies in
  $\mathcal{M}_b^*(R_k)\subset\mathcal{M}^*_b$ and we have $I(u_k,v_k)=\kappa_b^*(R_k)\to \kappa_b^*$ as
  $k\to\infty$ by Proposition \ref{Mbrad Prop Mbrad-Minimierung n=1 Approx durch Kugeln}.
  Since $(u_k,v_k)$ solves
  \eqref{Mbrad Gl DGL uR vR} on $(-R_k,R_k)$ there is a real number $\alpha_k$ such that
     \begin{align} \label{Mbrad Gl Existenzres b kleiner 0 n gleich 1 I}
       -u_k'^2 - v_k'^2 + u_k^2 + \omega^2 v_k^2
       - \frac{1}{q} (u_k^{2q}+v_k^{2q} + 2bu_k^qv_k^q) = \alpha_k \qquad \text{in } (-R_k,R_k)
     \end{align}
   and $u_k(R_k)=v_k(R_k)=0$ implies $\alpha_k\leq 0$, see Proposition \ref{Mbrad Prop Identität n=1} for the
   proof of a related result. The sequence $(u_k,v_k)$ is bounded in $H_r^1(\R)\times H_r^1(\R)$ and we may
   choose a subsequence again denoted by $(u_k,v_k)$ that converges weakly to some $(u,v)\in H^1_r(\R)\times
   H^1_r(\R)$. Then $R_k\to\infty$ implies that $(u,v)$ is a nonnegative solution of \eqref{Mbrad Gl DGL}
   with $I(u,v)\leq \kappa_b^*$. It remains to show $u,v\neq 0$.
   \medskip

   We assume $u=0$. Since $u_k,v_k$ are  radially symmetric we have $u_k'(0)=v_k'(0)=0$. From $b\leq 0$ and
   \eqref{Mbrad Gl Existenzres b kleiner 0 n gleich 1 I} we get the inequality
   $$
     0\geq \alpha_k \geq u_k(0)^2 (q-u_k(0)^{2q-2}) +  v_k(0)^2 (q\omega^2 -v_k(0)^{2q-2}).
   $$
   From $u_k(0)\to u(0)=0$ it follows $v_k(0)\geq (q\omega^2)^{\frac{1}{2q-2}}$ for almost all $k$
   and hence $v(0)>0$. It follows that $(u,v)=(0,v)$ is a solution of \eqref{Mbrad Gl DGL} satisfying
   $v(0)>0$ as well as $v\geq 0$. Kwong's uniqueness result \cite{Kwo_Uniqueness_of_positive} gives $v=v_0$
   and we obtain
     \begin{equation} \label{Mbrad Gl Konvergenz uk vk}
       (u_k,v_k) \wto (0,v_0),\qquad (u_k,v_k) \to (0,v_0) \text{ in }C^2_{loc}(\R).
     \end{equation}
     \medskip

   Let now $x_k\in [0,R_k)$ be given by
   $$
     \max_{[-R_k,R_k]} u_k = u_k(x_k)=u_k(-x_k)>0.
   $$
   From the differential equation \eqref{Mbrad Gl DGL uR vR} and $b\leq 0$ we infer
   $$
     0
     \leq - \frac{u_k''(x_k)}{u_k(x_k)}
     =  \frac{u_k(x_k)^{2q-1}+bu_k(x_k)^{q-1}v_k(x_k)^q-u_k(x_k)}{u_k(x_k)}
     \leq  u_k(x_k)^{2q-2}  - 1
   $$
   and thus
   \begin{equation} \label{Mbrad Gl Hoehe Max uk}
     u_k(x_k)=u_k(-x_k)\geq 1.
   \end{equation}
   From \eqref{Mbrad Gl Konvergenz uk vk},\eqref{Mbrad Gl Hoehe Max uk} we get $x_k\to +\infty$. Let
   now  $(\tilde u_k^+,\tilde v_k^+),(\tilde u_k^-,\tilde v_k^-)$ be given by
   $$
       (\tilde u_k^+,\tilde v_k^+):=(u_k(\cdot + x_k),v_k(\cdot +x_k)),\qquad
       (\tilde u_k^-,\tilde v_k^-):=(u_k(\cdot -x_k),v_k(\cdot - x_k)).
     $$
   These sequences are bounded in $H^1(\R)\times H^1(\R)$ and there are subsequences again denoted by
   $(\tilde u_k^+,\tilde v_k^+),(\tilde u_k^-,\tilde v_k^-)$ that converge weakly and locally uniformly to
   nonnegative functions $(\tilde u^+,\tilde v^+),(\tilde u^-,\tilde v^-)$, respectively.
   The inequality \eqref{Mbrad Gl Hoehe Max uk} implies $\tilde u^+(0), \tilde u^-(0)>0$.
   Since the functions $(\tilde u^+,\tilde v^+),(\tilde u^-,\tilde v^-)$ are nontrivial
   solutions of \eqref{Mbrad Gl DGL} on $(-\infty,a),(-a,\infty)$, respectively where $a:= \lim_{k\to\infty}
   (R_k-x_k)$ we obtain $(\tilde u^\pm,\tilde v^\pm)\in\mathcal{N}_b$ and \eqref{Gl für (u,v) in Nehari} gives
     \begin{equation} \label{Mbrad Gl Energieabschaetzung tildeu tilde v}
	   \frac{q-1}{2q}(\|\tilde u^\pm\|^2 + \|\tilde v^\pm\|_\omega^2)
	    = I(\tilde u^\pm,\tilde v^\pm)
       \geq \min_{\mathcal{N}_b} I
       = I(u_0,0)
       = c_0.
     \end{equation}
     \medskip

   Now let $\chi$ denote a cut-off-function satisfying $\chi(x)=1$ for $|x|\leq 1$ and $\chi(x)=0$ for
   $|x|\geq 2$, set $\chi_R(x):=\chi(R^{-1}x)$. Choosing $k_0(R)$ sufficiently large we
   obtain $x_k>2R$ for all $k\geq k_0(R)$. In particular  for all $k\geq
   k_0(R)$ the sets $\supp(\chi_R)$, $\supp(\chi_R(\cdot-x_k))$, $\supp(\chi_R(\cdot+x_k))$ are pairwise
   disjoint and we get
     \begin{align*}
       &\|(u_k,v_k)\|^2 \\
       &\geq \|(u_k\chi_R,v_k\chi_R)\|^2 + \|(u_k\chi_R(\cdot-x_k),v_k\chi_R(\cdot-x_k))\|^2
         + \|(u_k\chi_R(\cdot+x_k),v_k\chi_R(\cdot+x_k))\|^2   \\
       &= \|(u_k\chi_R,v_k\chi_R)\|^2 + \|(\tilde u_k^+\chi_R,\tilde v_k^+\chi_R)\|^2
         + \|(\tilde u_k^-\chi_R,\tilde v_k^-\chi_R)\|^2.
     \end{align*}
     From $(u_k,v_k)\wto (0,v_0), (\tilde u_k^+,\tilde v_k^+)\wto (\tilde u^+,\tilde v^+)$ and
     $(\tilde u_k^+,\tilde v_k^+)\wto (\tilde u^-,\tilde v^-)$  we infer
     \begin{align*}
       \liminf_{k\to\infty} \|(u_k,v_k)\|^2
       &\geq \|(0,v_0\chi_R)\|^2 +
         \|(\tilde u^+\chi_R,\tilde v^+\chi_R)\|^2 +  \|(\tilde u^-\chi_R,\tilde v^-\chi_R)\|^2  \\
       &= \|v_0\chi_R\|_\omega^2 +  \|(\tilde u^+\chi_R,\tilde v^+\chi_R)\|^2 +  \|(\tilde u^-\chi_R,\tilde
       v^-\chi_R)\|^2.
     \end{align*}
     Since this inequality holds for all $R>0$ we obtain
     \begin{align*}
       \liminf_{k\to\infty} \|(u_k,v_k)\|^2
       \geq \|v_0\|_\omega^2 + \|(\tilde u^+,\tilde v^+)\|^2 +  \|(\tilde u^-,\tilde   v^-)\|^2
     \end{align*}
     and from the estimate \eqref{Mbrad Gl Energieabschaetzung tildeu tilde v} and \eqref{Gl für (u,v) in
     Nehari} we get
     \begin{align*}
       \kappa_b^*
       &= \lim_{k\to\infty} \kappa_b^*(R_k) \\
       &= \frac{q-1}{2q} \lim_{k\to\infty} \|(u_k,v_k)\|^2 \\
       &\geq \frac{q-1}{2q} \Big(\|v_0\|_\omega^2 + \|(\tilde u^+,\tilde v^+)\|^2 +  \|(\tilde u^-,\tilde
       v^-)\|^2\Big) \\
 	   &\geq (2+\omega^{\frac{q+1}{q-1}})c_0
 	 \end{align*}
     which contradicts \eqref{Gl Ungl kappab kappa-infty}. Hence, $u\neq 0$. Analogously the
     assumption $v=0$ leads to the inequality
     $$
       \kappa_b^*
       \geq (1+2\omega^{\frac{q+1}{q-1}})c_0
       \geq (2+\omega^{\frac{q+1}{q-1}})c_0,
     $$
     which again gives a contradiction. It follows $u,v\neq 0$ and the proof is finished. \qed
 \medskip
 \medskip

  {\it Proof of Corollary \ref{Mbrad Kor hinreichendes Krit Mbrad n=1}}
  \medskip

  Assume that $b$ is larger than the right hand side in \eqref{Mbrad Gl hinreichendes Krit Mbrad n=1}.
  According to Theorem \ref{Mbrad Thm Kompaktheit n=1} (iii) it suffices to show that this implies
  $b>b^*(\omega,q)$. To this end we estimate $b^*(\omega,q)$ from above using the test function
  $(u,v):=(u_0,u_0(\omega\cdot))$ in \eqref{Gl Def b*(omega,q)}. We obtain
     \begin{align*}
        b^*(\omega,q)
        &\leq \max_{\alpha>0} \frac{(2+\omega^{\frac{q+1}{q-1}})^{1-q} \|u_0\|^{-2q}\|u_0\|_{2q}^{2q}
        (\|u_0\|^2+\alpha^2\|u_0(\omega\cdot)\|_\omega^2)^q-
        \|u_0\|_{2q}^{2q}-\alpha^{2q}\|u_0(\omega\cdot)\|_{2q}^{2q}}{2\alpha^q \|u_0u_0(\omega\cdot)\|_q^q} \\
        &=\max_{\alpha>0}  \frac{(2+\omega^{\frac{q+1}{q-1}})^{1-q}( 1+\alpha^2\omega)^q -1 -
        \alpha^{2q}\omega^{-1}}{2\alpha^q}\cdot \frac{\|u_0\|_{2q}^{2q}}{\|u_0u_0(\omega\cdot)\|_q^q}.
     \end{align*}
  The numerator function is bounded from above by its negative maximum $2^{1-q}-1$ which is attained at
  $\alpha=2^{-1/2}\omega^{\frac{1}{q-1}}$.
  In particular, the right hand side is negative for all $\alpha>0$ so that the estimate
  ${\|u_0u_0(\omega\cdot)\|_q^q\leq
  \|u_0\|_{2q}^q\|u_0(\omega\cdot)\|_{2q}^q=\omega^{-1/2}\|u_0\|_{2q}^{2q}}$ leads to
     \begin{align*}
         b^*(\omega,q)
         &\leq \max_{\alpha>0}  \frac{(2+\omega^{\frac{q+1}{q-1}})^{1-q}( 1+\alpha^2\omega)^q -1 -
        \alpha^{2q}\omega^{-1}}{2\alpha^q\omega^{-1/2}}
   \end{align*}
   where the right hand side is smaller than $b$ by the assumption of Corollary \ref{Mbrad Kor hinreichendes
   Krit Mbrad n=1}. As indicated above the result now follows from Theorem \ref{Mbrad Thm Kompaktheit
   n=1} (iii).
   \medskip

   Finally, in the special cases $q=2$ or $1<q\leq 2,\omega=1$ we may determine the value of the
   right hand side in \eqref{Mbrad Gl hinreichendes Krit Mbrad n=1} explicitly. In case $q=2$
   the maximum is attained at $\alpha = \big(\frac{(1+\omega^3)\omega}{2}\big)^{1/4}$ and
   we get
   $$
     b^*(\omega,2)\leq -\frac{1}{\omega^{3/2}+\sqrt{2(1+\omega^3)}}.
   $$
   In case $1<q\leq 2,\omega=1$ the maximum is attained at $\alpha=1$ and we obtain the value
   $$
     b^*(1,q)\leq (\frac{2}{3})^{q-1}-1.
   $$
   \medskip

  \section{Proof of Theorem \ref{Mbrad Thm Nichtexistenz n=1 q leq 2}}

    In the proof of Theorem \ref{Mbrad Thm Nichtexistenz n=1 q leq 2} we will need the following elementary
    result.

   \begin{prop}\label{Mbrad Prop Identität n=1}
     Let $n=1,\omega\geq 1$. Then every solution $(u,v)\in H^1(\R)\times H^1(\R)$ of \eqref{Mbrad Gl DGL}
     satisfies
     \begin{align}\label{Mbrad Gl Identitaet n=1}
       - u'^2 - v'^2 + u^2 + \omega^2 v^2
       - \frac{1}{q} (|u|^{2q}+|v|^{2q} + 2b|u|^q|v|^q) = 0 \qquad \text{in } \R.
     \end{align}
   \end{prop}
   \begin{proof}
     For a solution $(u,v)$  of \eqref{Mbrad Gl DGL} the derivative of the left hand side in \eqref{Mbrad Gl
     Identitaet n=1} exists and equals zero. Hence there is some $\alpha\in\R$ such that
     $$
       - u'^2 - v'^2 + u^2 + \omega^2 v^2
       - \frac{1}{q}(|u|^{2q}+|v|^{2q} + 2b|u|^q|v|^q) = \alpha \qquad \text{in } \R.
     $$
     If $\alpha$ were not equal to zero then there would exist $\delta>0$ such that
     $u'^2+v'^2+u^2+\omega^2 v^2 \geq \delta$ in $\R$ which contradicts $u,v\in H^1(\R)$.
   \end{proof}

   {\it Proof of Theorem \ref{Mbrad Thm Nichtexistenz n=1 q leq 2}}
   \medskip

    Let $b\in\R$ satisfy the inequality \eqref{Gl Nichtexistenzkriterium n=1}.
    We assume that there is a fully nontrivial solution $(u,v)\in H^1(\R)\times H^1(\R)$ of \eqref{Mbrad Gl
    DGL}. Since the functions $(-u,v),(u,-v),(-u,-v)$ solve \eqref{Mbrad Gl DGL}, too, we may assume that a
    maximal open interval $A\subset \{x\in\R: u(x)>0,v(x)>0\}$ is non-empty. We will prove later that the
    assumptions of the theorem imply that every critical point of $u^\omega v$ in $A$ is strict local minimizer.
    Once this is shown a contradiction can be achieved in the following way.
    \medskip

    In case $u^\omega v$ does not have any critical point in $A$ the function $u^\omega v$ is monotone on $A$
    so that $A$ is unbounded and $(u^\omega v)(x)$ does not converge to 0 as $|x|\to\infty$. This contradicts
    $u,v\in H^1(\R)$. If, however, a critical point $x_0\in A$ exists, then $x_0$ is a strict local
    minimizer and therefore it must be the only critical point because any other critical point would have
    to be a strict local minimizer, too. It follows that $u^\omega v$ is increasing on $(x_0,\infty)\cap A$
    and decreasing on $(-\infty,x_0)\cap A$ so that $(u^\omega v)(x)\geq (u^\omega v)(x_0)>0$ for all $x\in
    A$. Hence, $A=\R$ from the maximality of $A$ and thus $(u^\omega v)(x)\geq (u^\omega v)(x_0)>0$ for all
    $x\in\R$ which contradicts $u,v\in H^1(\R)$.
    \medskip

    Now we show that every critical point of $u^\omega v$ in $A$ is a strict local minimizer.
    Clearly, for $x\in A$ such that $(u^\omega v)'(x)=0$ we have
    \begin{equation} \label{Mbrad n=1 Gl 1}
      \omega u'(x)v(x) = -u(x)v'(x)
    \end{equation}
    and a short calculation gives
    \begin{equation} \label{Mbrad n=1 Gl 2}
      u'(x)v'(x) = -  \omega u(x)v(x) \cdot \frac{u'(x)^2+v'(x)^2}{u(x)^2+\omega^2 v(x)^2}.
    \end{equation}
    Using \eqref{Mbrad n=1 Gl 1} we obtain at the point $x$
    \begin{align*}
      (u^\omega v)''
      &= \omega(\omega-1) u^{\omega-2}u'^2 v + \omega u'' u^{\omega-1} v + 2\omega u'u^{\omega-1}v' + u^\omega v''
      \\
      &= u^{\omega-1}\Big(\omega\frac{u'}{u}\cdot (\omega-1)u'v + \omega u'' v + 2\omega u'v' + uv''\Big) \\
      &= u^{\omega-1}\Big( -(\omega-1)u'v' + \omega u'' v +
      + 2\omega u'v' + uv''\Big)
      \\
      &= u^{\omega-1}\Big((\omega+1)u'v' + \omega u'' v + uv''\Big).
      \intertext{
    From \eqref{Mbrad n=1 Gl 2} and the partial differential equation satisfied by $(u,v)$ we get}
      (u^\omega v)''
      &= u^{\omega-1}\Big(
      -\omega(\omega+1)uv \cdot \frac{u'^2+v'^2}{u^2+\omega^2 v^2}  \\
      &\qquad\quad + \omega
      uv(1-u^{2q-2}-bu^{q-2}v^q) + uv(\omega^2 -v^{2q-2}-bv^{q-2}u^q) \Big).
       \intertext{Proposition \ref{Mbrad Prop Identität n=1} gives}
      (u^\omega v)''
      &=
      u^{\omega}v\Big( -\omega(\omega+1)\cdot \Big( 1 -
      \frac{u^{2q}+v^{2q}+2bu^qv^q}{q(u^2+\omega^2 v^2)}\Big) \\
      &\qquad\quad  + \omega - \omega u^{2q-2}-b\omega u^{q-2}v^q +  \omega^2 -v^{2q-2}-bv^{q-2}u^q \Big) \\
      &= \frac{u^\omega v}{q(u^2+\omega^2 v^2)}\cdot  \Big( \omega(\omega+1)
      (u^{2q}+v^{2q}+2bu^qv^q)  -  q(u^2+\omega^2 v^2)\cdot \\
      &\qquad\qquad\qquad\qquad  \big( \omega u^{2q-2} +b\omega u^{q-2}v^q + v^{2q-2} + bv^{q-2}u^q\big)\Big)
      \\
      &=  \frac{u^\omega v}{q(u^2+\omega^2 v^2)}  \cdot \Big(
        -b q u^{q+2}v^{q-2}
      + (\omega^2-(q-1)\omega) u^{2q}
      - q u^2v^{2q-2} \\
      &\qquad\qquad\qquad\qquad
      +b(2-q)(\omega^2+\omega) u^qv^q
      -q\omega^3 u^{2q-2}v^2 \\
      &\qquad\qquad\qquad\qquad
      - (\omega^2(q-1)-\omega) v^{2q}
      - bq\omega^3 u^{q-2}v^{q+2} \Big) \\
      &=  \frac{u^\omega v^{2q+1}}{q(u^2+\omega^2 v^2)}  \cdot \Big(
        -b q z^{q+2}
      + (\omega^2-(q-1)\omega) z^{2q}
      - q z^2 \\
      &\qquad\qquad\qquad\qquad
      +b(2-q)(\omega^2+\omega) z^q
      -q\omega^3 z^{2q-2} \\
      &\qquad\qquad\qquad\qquad
      - (\omega^2(q-1)-\omega)
      - bq\omega^3 z^{q-2} \Big)
    \end{align*}
    where $z:= \frac{u(x)}{v(x)}$. From assumption \eqref{Gl Nichtexistenzkriterium n=1} and $u(x),v(x),z>0$
    we obtain $(u^\omega v)''(x)>0$ which proves the claim. We finish the proof of
    Theorem \ref{Mbrad Thm Nichtexistenz n=1 q leq 2} considering the special cases $q=2$ and
    $1<q<2,\omega=1$.
    \medskip

    In case $q=2$ the minimum in \eqref{Gl Nichtexistenzkriterium n=1} is attained at $z=\sqrt{\omega}$ and we
    obtain that fully nontrivial solutions do not exist for parameter values
    $b < - \frac{\omega^2+1}{2\omega}$.
    In case $1<q<2,\omega=1$ we find that the following inequality holds for $b\leq -1$ and all $z>0$
    \begin{align*}
      (-b)\cdot \big(q z^{q+2} &-2(2-q)z^q + q z^{q-2}\big) + (2-q) z^{2q}  - q z^2
      - q z^{2q-2}  - (q-2) \\
      &\geq 1\cdot \big(q z^{q+2} -2(2-q)z^q + q z^{q-2}\big) + (2-q) z^{2q}  - q z^2
      - q z^{2q-2}  - (q-2) \\
      &= (2-q)(z^q-1)^2 + q(z^2-z^{q-2})(z^q-1) \\
      &\geq 0
    \end{align*}
    with equality if and only if $b=-1$ and $z=1$. Rearranging terms we see that the minimum in
    \eqref{Gl Nichtexistenzkriterium n=1} is 1 and it is attained at $z=1$. We obtain the nonexistence result
    for $b<-1$.\qed

  \begin{bem}
    In the above reasoning we did not use the assumption $1<q\leq 2$ explicitly. Nevertheless we had
    to exclude the case $q>2$ case because the minimum in \eqref{Gl Nichtexistenzkriterium n=1} does not
    exist. Indeed, sending $z$ to $0$ and using $\omega\geq 1>\frac{1}{q-1}$ we find that the infimum is
    $-\infty$.
  \end{bem}

\bibliographystyle{plain}
\bibliography{pub}

\begin{thebibliography}{10}

\bibitem{AmbCol_Bound_and_Ground}
A.~Ambrosetti and E.~Colorado.
\newblock Bound and ground states of coupled nonlinear {S}chr\"odinger
  equations.
\newblock {\em C. R. Math. Acad. Sci. Paris}, 342(7):453--458, 2006.

\bibitem{BaDaWa_A_Liouville_Theorem}
T.~Bartsch, N.~Dancer, and Z.-Q. Wang.
\newblock A {L}iouville theorem, a-priori bounds, and bifurcating branches of
  positive solutions for a nonlinear elliptic system.
\newblock {\em Calc. Var. Partial Differential Equations}, 37(3-4):345--361,
  2010.

\bibitem{CoFe_Minimal_coexistence}
M.~Conti and V.~Felli.
\newblock Minimal coexistence configurations for multispecies systems.
\newblock {\em Nonlinear Anal.}, 71(7-8):3163--3175, 2009.

\bibitem{CoFe_Global_minimizers_of}
M.~Conti and V.~Felli.
\newblock Global minimizers of coexistence for competing species.
\newblock {\em J. Lond. Math. Soc. (2)}, 83(3):606--618, 2011.

\bibitem{CoTeVe_A_variational_problem}
M.~Conti, S.~Terracini, and G.~Verzini.
\newblock A variational problem for the spatial segregation of
  reaction-diffusion systems.
\newblock {\em Indiana Univ. Math. J.}, 54(3):779--815, 2005.

\bibitem{DeFLop_Solitary_waves_for}
D.~G. de~Figueiredo and O.~Lopes.
\newblock Solitary waves for some nonlinear {S}chr\"odinger systems.
\newblock {\em Ann. Inst. H. Poincar\'e Anal. Non Lin\'eaire}, 25(1):149--161,
  2008.

\bibitem{DoUe_Coupled_mode_equations}
T.~{Dohnal} and H.~{Uecker}.
\newblock {Coupled mode equations and gap solitons for the 2D Gross-Pitaevskii
  equation with a non-separable periodic potential}.
\newblock {\em Physica D Nonlinear Phenomena}, 238:860--879, 2009.

\bibitem{Kwo_Uniqueness_of_positive}
M.~K. Kwong.
\newblock Uniqueness of positive solutions of {$\Delta u-u+u^p=0$} in {${\bf
  R}^n$}.
\newblock {\em Arch. Rational Mech. Anal.}, 105(3):243--266, 1989.

\bibitem{LiWe_Ground_states_of}
T.-C. Lin and J.~Wei.
\newblock Ground state of {$N$} coupled nonlinear {S}chr\"odinger equations in
  {$\R^n$}, {$n\leq 3$}.
\newblock {\em Comm. Math. Phys.}, 255(3):629--653, 2005.

\bibitem{LiWe_Erratum_ground_states}
T.-C. Lin and J.~Wei.
\newblock Erratum: ``{G}round state of {$N$} coupled nonlinear {S}chr\"odinger
  equations in {${\R}^n$}, {$n\leq3$}'' [{C}omm. {M}ath. {P}hys. 255 (2005),
  no. 3, 629--653; mr2135447].
\newblock {\em Comm. Math. Phys.}, 277(2):573--576, 2008.

\bibitem{MaMoPe_Positive_solutions_for}
L.~A. Maia, E.~Montefusco, and B.~Pellacci.
\newblock Positive solutions for a weakly coupled nonlinear {S}chr\"odinger
  system.
\newblock {\em J. Differential Equations}, 229(2):743--767, 2006.

\bibitem{Sir_Least_energy_solutions}
B.~Sirakov.
\newblock Least energy solitary waves for a system of nonlinear {S}chr\"odinger
  equations in {$\mathbb{R}^n$}.
\newblock {\em Comm. Math. Phys.}, 271(1):199--221, 2007.

\bibitem{TaTe_Sign-changin_solutions_of}
H.~Tavares and S.~Terracini.
\newblock Sign-changing solutions of competition-diffusion elliptic systems and
  optimal partition problems.
\newblock {\em Ann. Inst. H. Poincar\'e Anal. Non Lin\'eaire}, 29(2):279--300,
  2012.

\bibitem{WeWe_Nonradial_symmetric_bound}
J.~Wei and T.~Weth.
\newblock Nonradial symmetric bound states for a system of coupled
  {S}chr\"odinger equations.
\newblock {\em Atti Accad. Naz. Lincei Cl. Sci. Fis. Mat. Natur. Rend. Lincei
  (9) Mat. Appl.}, 18(3):279--293, 2007.

\bibitem{WeWe_Radial_solutions_and}
J.~Wei and T.~Weth.
\newblock Radial solutions and phase separation in a system of two coupled
  {S}chr\"odinger equations.
\newblock {\em Arch. Ration. Mech. Anal.}, 190(1):83--106, 2008.

\end{thebibliography}

\end{document}